\documentclass[a4paper,12pt,reqno]{amsart}
\usepackage{eurosym}
\usepackage{amsfonts}
\usepackage{amsmath,amsthm,amssymb}
\usepackage{graphicx, color}
\usepackage{cite}

\nonstopmode \numberwithin{equation}{section}

\newtheorem{theorem}{Theorem}
\newtheorem{corollary}{Corollary}[section]

\allowdisplaybreaks

\begin{document}
\title[ integrals associated with the product of generalized Struve function]{Certain new
unified integrals associated with the product of generalized Struve function}
\author[K. S.  Nisar]{Kottakkaran Sooppy  Nisar}
\address{ K. S. Nisar: Department of Mathematics, College of Arts and
Science-Wadi Al dawaser, Prince Sattam bin Abdulaziz University, Riyadh
region 11991, Saudi Arabia}
\email{ksnisar1@gmail.com, n.sooppy@psau.edu.sa}

%\thanks{$^{*}$ Corresponding author}

\begin{abstract}
We aim to present two new generalized integral formulae involving product
of generalized Struve function $\mathcal{W}_{p,b,c}\left( z\right)$,  which are
expressed in terms of the generalized Lauricella functions.
The main results presented here, being very general character,
reduce to yield known and new integral formulae. Some  special cases of our main results are also considered.
\end{abstract}

\maketitle

\section{Introduction and Preliminaries}\label{sec-1}

A  solution of the following non-homogeneous Bessel's differential equation
\begin{equation}\label{StruveH}
x^{2}y^{''}(x)+xy^{'}\left( x\right) +\left(
x^{2}-p^{2}\right) y\left( x\right) =\frac{4\,(x/2)^{\alpha +1}}{\sqrt{\pi }\,
\Gamma \left( {\alpha +1/2}\right) }\quad (\alpha \in \mathbb{C})
\end{equation}
is known as the Struve function $\mathcal{H}_{\alpha }$ of order $\alpha $,
where $\Gamma$ is the familiar gamma function (see, e.g., \cite[Section 1.1]{Sr-Ch-12}).
Here and in the following, let $\mathbb{C}$, $\mathbb{R}^+$, and $\mathbb{N}$ be the sets of complex numbers,
positive real numbers, and positive integers, respectively, and let $\mathbb{N}_0:=\mathbb{N} \cup \{0\}$.
The Struve functions occur in many areas such as water
wave and surface-wave problems (see \cite{Ahmadi-Widnall, Hirata}),
 problems on unsteady aerodynamics \cite{Shaw}, particle quantum dynamical
studies of spin decoherence \cite{Shao} and nanotubes \cite{Pedersen}. The
Struve functions $\mathcal{H}_{v}\left( z\right) $ and $\mathcal{L}_{v}\left( z\right) $ are
defined as  the following infinite series
\begin{equation} \label{Struve}
\mathcal{H}_{\upsilon }\left( z\right) =\left( \frac{z}{2}\right) ^{\upsilon
+1}\sum\limits_{k=0}^{\infty }\frac{\left( -1\right) ^{k}}{\Gamma \left( k+
\frac{3}{2}\right) \Gamma \left( k+\upsilon +\frac{1}{2}\right) }\left(
\frac{z}{2}\right) ^{2k}
\end{equation}
and
\begin{equation}\label{MStruve}
\mathcal{L}_{\upsilon }\left( z\right) =\left( \frac{z}{2}\right) ^{\upsilon
+1}\sum\limits_{k=0}^{\infty }\frac{1}{\Gamma \left( k+\frac{3}{2}\right)
\Gamma \left( k+\upsilon +\frac{1}{2}\right) }\left( \frac{z}{2}\right) ^{2k}.
\end{equation}

The generalized Struve function is given as follows (see, e.g.,  \cite{Nis-Pur-Mond}):
\begin{equation}\label{GStruve}
\mathcal{W}_{p,b,c}\left( z\right) =\sum_{k=0}^{\infty }\frac{\left(
-c\right) ^{k}\left( \frac{z}{2}\right) ^{2k+p+1}}{\Gamma \left( k+\frac{3}{2
}\right) \Gamma \left( k+p+\frac{b+2}{2}\right) } \quad (p,\, b,\, c\in \mathbb{C}).
\end{equation}

An interesting further generalization of the generalized hypergeometric series ${}_pF_q$ (see, e.g., \cite[Section 1.5]{Sr-Ch-12}) is due to Fox \cite{Fox}
and Wright \cite{Wright1, Wright2} who studied the asymptotic expansion of the generalized (Wright) hypergeometric function
defined by (see \cite[p. 21]{Sri-Kar})
\begin{equation} \label{Fox-Wright}
 {}_p\Psi_q \left[ \aligned \left(\alpha_1,\,A_1\right),\,\ldots,\,\left(\alpha_p,\,A_p\right) \, &;\\
                      \left(\beta_1,\,B_1\right),\,\ldots,\,\left(\beta_q,\,B_q\right) \,&; \endaligned
                  \,\, z \right] = \sum_{k=0}^\infty \,\frac{\prod\limits_{j=1}^p\, \Gamma \left( \alpha_j + A_j\,k\right)}{\prod\limits_{j=1}^q\, \Gamma \left( \beta_j + B_j\,k\right)}\, \frac{z^k}{k!},
    \end{equation}
where the coefficients $A_1,\,\ldots,\,A_p \in \mathbb{R}^+$ and $B_1,\,\ldots,\,B_q \in \mathbb{R}^+$ such that
\begin{equation}\label{Condi}
1 + \sum_{j=1}^q\, B_j - \sum_{j=1}^p\, A_j \geqq 0.
    \end{equation}
A special case of \eqref{Fox-Wright} is
\begin{equation} \label{Fox-Wright-a}
 {}_p\Psi_q \left[ \aligned \left(\alpha_1,\,1\right),\,\ldots,\,\left(\alpha_p,\,1\right) \, &;\\
                      \left(\beta_1,\,1\right),\,\ldots,\,\left(\beta_q,\,1\right) \,&; \endaligned
                  \,\, z \right] =\frac{\prod\limits_{j=1}^p\, \Gamma \left( \alpha_j\right)}{\prod\limits_{j=1}^q\, \Gamma \left( \beta_j\right)}\, {}_pF_q \left[ \aligned \alpha_1,\,\ldots,\,\alpha_p\, &;\\
                       \beta_1,\,\ldots,\,\beta_q \,&; \endaligned
                  \,\, z \right].
\end{equation}

\vskip 3mm

An interesting further several-variable-generalization of the generalized Lauricella series (see, for example, \cite[p. 36, Eq. (19)]{Sri-Kar})
is defined by (\emph{cf}. Srivastava and Daoust \cite[p. 454]{Sri-Dao}; see also \cite[p. 37]{Sri-Kar})
\begin{equation}\label{Lauricella-Def}
\aligned
& F_{C:D^{(1)};\cdots;D^{(n)}}^{A:B^{(1)};\cdots;B^{(n)}}\, \left(
                                                  \begin{array}{c}
                                                    z_1 \\
                                                    \vdots \\
                                                    z_n \\
                                                  \end{array}
                                                \right)
             =F_{C:D^{(1)};\cdots;D^{(n)}}^{A:B^{(1)};\cdots;B^{(n)}}
                \left(\aligned
            &\,\, [(a):\theta^{(1)},\ldots,\theta^{(n)}]:\\
            &[(c):\psi^{(1)},\ldots,\psi^{(n)}]:
            \endaligned\right.\\
              & \hskip 45mm   \left.
            \aligned
            &[(b)^{(1)}:\phi^{(1)}];\ldots;[(b)^{(n)}:\phi^{(n)}]; \\
            &[(d)^{(1)}:\delta^{(1)}];\ldots;[(d)^{(n)}:\delta^{(n)}];
            \endaligned \, z_{1},\ldots,z_{n}\right)\\
        &  \hskip 35mm   =\sum_{k_{1},\ldots,k_{n}=0}^{\infty}\, \Omega (k_{1},\ldots,k_{n})\,
                \frac{z_{1}^{k_{1}}}{k_{1}!}\cdots
                     \frac{z_{n}^{k_{n}}}{k_{n}!},
        \endaligned
      \end{equation}
where, for convenience,
\begin{equation}\label{Omega} \Omega(k_{1},\ldots,k_{n})=\frac{\displaystyle{\prod_{j=1}^{A}}(a_{j})_{k_{1}\theta^{(1)}_{j}+\cdots+k_{n}\theta_{j}^{(n)}}
            \,\prod_{j=1}^{B^{(1)}}(b^{(1)}_{j})_{k_{1}\phi^{(1)}_{j}}\cdots\prod_{j=1}^{B^{(n)}}(b_{j}^{(n)})_{k_{n}\phi_{j}^{(n)}}}
            {\displaystyle{\prod_{j=1}^{C}}(c_{j})_{k_{1}\psi^{(1)}_{j}+\cdots+k_{n}\psi_{j}^{(n)}}
            \prod_{j=1}^{D^{(1)}}(d^{(1)}_{j})_{k_{1}\delta^{(1)}_{j}}\cdots\prod_{j=1}^{D^{(n)}}(d_{j}^{(n)})_{k_{n}\delta_{j}^{(n)}}},
   \end{equation}
the coefficients
\begin{equation}\label{Condition}
\left\{\aligned
&\theta_{j}^{(m)}\,\,(j=1,\ldots,A); \,\,\phi_{j}^{(m)}\,\,(j=1,\ldots,B^{(m)});\\
&\psi_{j}^{(m)}\,\,(j=1,\ldots,C);\,\,\delta_{j}^{(m)}\,\,(j=1,\cdots,D^{(m)}); \,\, \forall \, m  \in \{1,\ldots,n \}
\endaligned  \right.
\end{equation}
are real and positive, and $(a)$ abbreviates the array of $A$  parameters $a_{1},\ldots,a_{A},$ $(b^{(m)}) $ abbreviates the array of $B^{(m)}$
parameters
$$ b_{j}^{(m)} \quad (j=1,\ldots,B^{(m)}); \quad \forall \,  m  \in \{1,\ldots,n\},$$
 with similar interpretations for $(c)$ and $(d^{(m)})\,\,(m=1,\ldots,n)$; \emph{et cetera}.

For more details about the generalized Lauricella function \eqref{Lauricella-Def},
the reader may be referred, for example,  to \cite{Shukla, Sri-Kar, Sri-Dao, Anvar-Sri, Choi-Nisar}.

\vskip 3mm

We also need the following integral formula due to Oberhettinger \cite{Ober}
\begin{equation}\label{Ober-formula}
\int_{0}^{\infty}\,x^{\mu -1}\left( x+a+\sqrt{x^{2}+2ax}\right) ^{-\lambda
}dx=2\lambda a^{-\lambda }\left( \frac{a}{2}\right) ^{\mu }\frac{\Gamma
\left( 2\mu \right) \Gamma \left( \lambda -\mu \right) }{\Gamma \left(
1+\lambda +\mu \right) }
\end{equation}
\begin{equation*}
  \left(a \in \mathbb{R}^+;\,  0<\Re (\mu)<\Re (\lambda)\right).
\end{equation*}

Many integral formulas involving Bessel functions, generalized Bessel functions, and modified Bessel functions,
and so on, have recently been established (see, e.g., \cite{Praveen,Ali,Choi2,Choi-Mathur,Nisar-Saiful,Nisar-Parmar}).
In this sequel, we aim to present two new integral formulas involving a finite product of the generalized Struve
functions \eqref{GStruve}, which are expressed in terms of the generalized Lauricella function \eqref{Lauricella-Def}. The main formulas presented here, being very general,  reduce to yield
known and new integral formulas. Some interesting special cases are also considered.

\section{Main results}\label{sec-2}

Here, we establish two integral formulas involving a finite product of the generalized Struve
functions \eqref{GStruve}, which are  asserted by Theorems \ref{Th1} and \ref{Th2}.

\vskip 3mm

\begin{theorem} \label{Th1}
Let $n \in \mathbb{N}$ be fixed and $a,\, y_j \in \mathbb{R}^+$ $(j=1,\ldots,n)$.
Also, let $b,\,c,\,p_j,\, \lambda,\,\mu \in \mathbb{C}$ $(j=1,\ldots,n)$ with
$0<\Re(\mu)<\Re \left(\lambda + \mathbf{p}  \right) +n$, where
\begin{equation}\label{p}
  \mathbf{p}:= \sum_{j=1}^{n}\,p_j.
\end{equation}
Then
\begin{equation}\label{Th1-eq1}
 \aligned
   &\int_{0}^{\infty }x^{\mu -1}\left( x+a+\sqrt{x^{2}+2ax}\right) ^{-\lambda
}\prod\limits_{j=1}^{n}\mathcal{W}_{p_{j},b,c}\left( \frac{y_{j}}{x+a+\sqrt{
x^{2}+2ax}}\right)\, dx \\
 &= \mathcal{A} (a,\mathbf{p}, \lambda, \mu, p_j,y_j:j=1,\ldots,n)
 \\
&\hskip 5mm \times F^{2:1;\cdots;1}_{\,2:2;\cdots;2}\, \bigg(
    \aligned
      & [1+\lambda+\mathbf{p}+n:2,\ldots,2], [\lambda+\mathbf{p}+n-\mu:2,\ldots,2] :  \\
      & [\lambda+\mathbf{p}+n:2,\ldots,2], [1+\lambda+\mathbf{p}+n+\mu:2,\ldots,2]:
      \endaligned  \\
  & \hskip 5mm   \aligned
      & \hskip 65mm  [1:1]; \cdots ; [1:1]; \\
    & \Big[\frac{3}{2}:1\Big], \Big[p_1 + \frac{b+2}{2}:1\Big];\cdots  ;\Big[\frac{3}{2}:1\Big],
     \Big[p_n + \frac{b+2}{2}:1\Big];
   \endaligned  -\frac{c\,y_1^2}{4\,a^2},\cdots, -\frac{c\,y_n^2}{4\,a^2}   \bigg),
 \endaligned
\end{equation}
where $1;\cdots;1$ and $2;\cdots;2$ are $n$ arrays, respectively, and
   \begin{equation*}
\aligned &
    \mathcal{A} (a,\mathbf{p}, \lambda, \mu, p_j,y_j:j=1,\ldots,n) \\
    & := \frac{(\lambda+ \mathbf{p}+n)\, 2^{1-\mu-\mathbf{p}-n}\,a^{\mu-\lambda-\mathbf{p}-n}\, \Gamma (2 \mu)\,\Gamma (\lambda+\mathbf{p}+n-\mu)\,
      \prod\limits_{j=1}^{n} {y_j}^{p_j+1}}{\left\{\Gamma\left(\frac{3}{2}\right)\right\}^n\,
     \Gamma (1+\lambda+\mathbf{p}+n+\mu)\,\prod\limits_{j=1}^{n} \Gamma \left(p_j + \frac{b+2}{2}\right) }.
\endaligned
  \end{equation*}

\end{theorem}

\begin{proof}
Let $\mathcal{L}$ be the left-hand side of \eqref{Th1-eq1}.
By using the generalized Struve function \eqref{GStruve} and changing the order of integration
and summations, which is verified under the given conditions of this theorem, we have
\begin{equation}\label{Th1-pf1}
  \aligned
 \mathcal{L} = & \sum_{k_1=0}^{\infty}\, \frac{(-c)^{k_1}\, \left(\frac{y_1}{2}\right)^{2k_1+p_1+1}}
 {\Gamma\left(k_1+\frac{3}{2}\right)\, \Gamma\left(k_1+p_1 + \frac{b+2}{2}\right)}
  \cdots \sum_{k_n=0}^{\infty}\, \frac{(-c)^{k_n}\, \left(\frac{y_n}{2}\right)^{2k_n+p_n+1}}
 {\Gamma\left(k_n+\frac{3}{2}\right)\, \Gamma\left(k_n+p_n + \frac{b+2}{2}\right)} \\
& \times \int_{0}^{\infty }x^{\mu -1}\left( x+a+\sqrt{x^{2}+2ax}\right) ^{-\lambda-(p_1+\cdots+p_n)-2(k_1+\cdots+k_n)-n}
    \,dx.
  \endaligned
\end{equation}
By using \eqref{Ober-formula} to evaluate the integral in \eqref{Th1-pf1} and interpreting  the resulting expression
in terms of the   Pochhammer symbol  defined (for $\lambda,\,\nu \in \mathbb{C}$) by
\begin{equation}\label{Poch-symbol}
  (\lambda)_\nu := \frac{\Gamma (\lambda +\nu)}{ \Gamma (\lambda)}
 =\left\{\aligned & 1  \hskip 49 mm (\nu=0;\,\, \lambda \in \mathbb{C}\setminus \{0\}) \\
        & \lambda (\lambda +1) \cdots (\lambda+n-1) \hskip 10mm (\nu=n \in {\mathbb N};\,\, \lambda \in \mathbb{C})
   \endaligned \right. \\
\end{equation}
together with the following easily-derivable identity:
\begin{equation}\label{Poch-id1}
  \lambda + k = \frac{\lambda \cdot (1+\lambda)_k }{(\lambda)_k} \quad \left(k \in \mathbb{N}_0\right),
\end{equation}
we obtain
\begin{equation*}
  \aligned
  \mathcal{L} = & \Lambda (a,\mathbf{p}, \lambda, \mu, p_j,y_j:j=1,\ldots,n) \\
    & \times
    \sum_{k_1,\ldots,k_n=0}^{\infty}\, \frac{(-c)^{k_1}\, \left(\frac{y_1}{2}\right)^{2k_1} \cdots (-c)^{k_n}\, \left(\frac{y_n}{2}\right)^{2k_n}\, a^{-2(k_1+\cdots+k_n)}  }
      {\left(\frac{3}{2}\right)_{k_1}\,  \left(p_1 + \frac{b+2}{2}\right)_{k_1} \cdots
  \left(\frac{3}{2}\right)_{k_n}\,  \left(p_n + \frac{b+2}{2}\right)_{k_n}  } \\
   &\hskip 5mm  \times \frac{(1+\lambda +\mathbf{p}+n)_{2(k_1+\cdots+k_n)}\,  (\lambda +\mathbf{p}+n-\mu)_{2(k_1+\cdots+k_n)} }
      {(\lambda +\mathbf{p}+n)_{2(k_1+\cdots+k_n)}\,  (1+\lambda +\mathbf{p}+n+\mu)_{2(k_1+\cdots+k_n)}},
   \endaligned
\end{equation*}
which, upon the multiple summations being expressed in terms of \eqref{Lauricella-Def}, leads to the right-hand side of
 \eqref{Th1-eq1}.

\end{proof}

\vskip 3mm

\begin{theorem} \label{Th2}
Let $n \in \mathbb{N}$ be fixed and $a,\, y_j \in \mathbb{R}^+$ $(j=1,\ldots,n)$.
Also, let $b,\,c,\,p_j,\, \lambda,\,\mu \in \mathbb{C}$ $(j=1,\ldots,n)$ with
  $\Re(\mu + \mathbf{p})>-n$ and $\Re(\lambda)>\Re(\mu)$.
Then
\begin{equation}\label{Th2-eq1}
 \aligned
   &\int_{0}^{\infty }x^{\mu -1}\left( x+a+\sqrt{x^{2}+2ax}\right) ^{-\lambda
}\prod\limits_{j=1}^{n}\mathcal{W}_{p_{j},b,c}\left( \frac{x\,y_{j}}{x+a+\sqrt{
x^{2}+2ax}}\right)\, dx \\
 &= \mathcal{B} (a,\mathbf{p}, \lambda, \mu, p_j,y_j:j=1,\ldots,n)
 \\
&\hskip 5mm \times F^{2:1;\cdots;1}_{\,2:2;\cdots;2}\, \bigg(
    \aligned
      &\hskip 5mm [2\mu+ 2\mathbf{p}+2n:4,\ldots,4], [1+\lambda+\mathbf{p}+n:2,\ldots,2] :  \\
      & [1+\lambda+\mu+2\mathbf{p}+2n:4,\ldots,4], [\lambda+\mathbf{p}+n:2,\ldots,2]:
      \endaligned  \\
  &  \hskip 5mm   \aligned
      & \hskip 65mm  [1:1]; \cdots ; [1:1]; \\
    & \Big[\frac{3}{2}:1\Big], \Big[p_1 + \frac{b+2}{2}:1\Big];\cdots  ;\Big[\frac{3}{2}:1\Big],
     \Big[p_n + \frac{b+2}{2}:1\Big];
   \endaligned  -\frac{c\,y_1^2}{16},\cdots, -\frac{c\,y_n^2}{16}   \bigg),
 \endaligned
\end{equation}
where $1;\cdots;1$ and $2;\cdots;2$ are $n$ arrays, respectively, and
  \begin{equation*}
\aligned &
    \mathcal{B} (a,\mathbf{p}, \lambda, \mu, p_j,y_j:j=1,\ldots,n) \\
    &    := \frac{(\lambda+ \mathbf{p}+n)\, 2^{1-\mu-2\mathbf{p}-2n }\,a^{\mu-\lambda}\, \Gamma (\lambda-\mu)\,\Gamma (2\mu+ 2\mathbf{p}+2n )\,
      \prod\limits_{j=1}^{n}y_j^{p_j+1} }{\left\{\Gamma\left(\frac{3}{2}\right)\right\}^n\,
     \Gamma (1+\lambda+\mu+ 2\mathbf{p}+2n)\,\prod\limits_{j=1}^{n} \Gamma \left(p_j + \frac{b+2}{2}\right) },
\endaligned
  \end{equation*}
and $\mathbf{p}$ is the same as in \eqref{p}.
\end{theorem}

\begin{proof}
The proof runs parallel to that of Theorem \ref{Th1}. We omit the details.
\end{proof}

\section{Special cases}\label{sec3}
Here, we consider some special cases of the main results. 

\vskip 3mm

Setting $n=1$ in the result in Theorem \ref{Th1}, we have an integral formula involving  the generalized Struve
function \eqref{GStruve}, which is given in Corollary \ref{cor1}.

\vskip 3mm

\begin{corollary}\label{cor1}
Let $a,\, y \in \mathbb{R}^+$.
Also, let $b,\,c,\,p,\, \lambda,\,\mu \in \mathbb{C}$ with
$0<\Re(\mu)<\Re \left(\lambda + p  \right) +1$.
Then
\begin{equation}\label{cor1-eq1}
  \aligned
  & \int_{0}^{\infty }x^{\mu -1}\left( x+a+\sqrt{x^{2}+2ax}\right) ^{-\lambda }
\mathcal{W}_{p,b,c}\left( \frac{y}{x+a+\sqrt{x^{2}+2ax}}\right)\, dx \\
  &= (1+\lambda+p)\,2^{-\mu-p}\,a^{\mu-1-\lambda-p}\,y^{p+1}\,
     \frac{\Gamma (2 \mu)\, \Gamma (1+\lambda+p-\mu)}{\Gamma \left(\frac{3}{2}\right)\,\Gamma (2+\lambda+p+\mu)\,
     \Gamma \left(1+\frac{b}{2}+p\right) } \\
   & \times {}_4F_5 \left[\begin{array}{r}
                         \frac{3}{2} + \frac{\lambda}{2}+\frac{p}{2},\, \frac{1}{2} + \frac{\lambda}{2}+\frac{p}{2}-\frac{\mu}{2},\, 1 +  \frac{\lambda}{2}+\frac{p}{2}-\frac{\mu}{2},\,1;\,  \\
          \frac{1}{2} + \frac{\lambda}{2}+\frac{p}{2},\, 1 +  \frac{\lambda}{2}+\frac{p}{2}+\frac{\mu}{2},\,\frac{3}{2} + \frac{\lambda}{2}+\frac{p}{2}+\frac{\mu}{2},\, 1+ \frac{b}{2}+p,\, \frac{3}{2};\,
                          \end{array}
  - \frac{c y^2}{4 a^2}   \right].
  \endaligned
\end{equation}

\end{corollary}

\vskip 3mm

Setting $n=1$ in the result in Theorem \ref{Th2}, we have another integral formula involving  the generalized Struve
function \eqref{GStruve}, which is given in Corollary \ref{cor2}.

\vskip 3mm

\begin{corollary}\label{cor2}
Let $a,\, y \in \mathbb{R}^+$.
Also, let $b,\,c,\,p,\, \lambda,\,\mu \in \mathbb{C}$ with
$\Re(\mu +p)>-1$ and $\Re(\lambda)>\Re(\mu)$.
Then
\begin{equation}\label{cor2-eq1}
\aligned
 & \int_{0}^{\infty }x^{\mu -1}\left( x+a+\sqrt{x^{2}+2ax}\right) ^{-\lambda }
\mathcal{W}_{p,b,c}\left( \frac{xy}{x+a+\sqrt{x^{2}+2ax}}\right) \,dx \\
&\hskip 5mm = 2^{-\mu-2p-1}\, a^{\mu-\lambda}\,y^{p+1}\, \Gamma (\lambda-\mu) \\
& \times {}_3 \Psi_4\left[ \aligned \left(1,\,1\right),\,\left(\lambda+p+2,\,2\right),\, \left(2\mu+2p+2,\,4\right)  \, &;\\
                      \left(\frac{3}{2},\,1\right),\,\left(p+\frac{b+2}{2},\,1\right),\,\left(\lambda+p+1,\,2\right),
    \,\left(\lambda+\mu +2p+3,\,4\right)    \,&; \endaligned
                  \,\, - \frac{c\,y^2}{16} \right].
\endaligned
\end{equation}

\end{corollary}

\vskip 3mm

It is easy to see from \eqref{Struve} and \eqref{GStruve} that 
\begin{equation}\label{H-W}
 \mathcal{W}_{p,-1,1}(z)=\mathcal{H}_p(z).
\end{equation}
Considering the relation \eqref{H-W} and setting $n=1$, $b=-1$, and $c=1$ in Theorems \ref{Th1} and \ref{Th2}
(or setting  $b=-1$ and $c=1$ in Corollaries \ref{cor1} and \ref{cor2}, we obtain two integral formulas
involving the Struve function $\mathcal{H}_p$ of order $p$, which are given, respectively, in Corollaries
 \ref{cor3} and \ref{cor4}.

\vskip 3mm 
\begin{corollary}\label{cor3}
Let $a,\, y \in \mathbb{R}^+$.
Also, let $p,\, \lambda,\,\mu \in \mathbb{C}$ with
$0<\Re(\mu)<\Re \left(\lambda + p  \right) +1$.
Then
\begin{equation}\label{cor3-eq1}
  \aligned
  & \int_{0}^{\infty }x^{\mu -1}\left( x+a+\sqrt{x^{2}+2ax}\right) ^{-\lambda }
\mathcal{W}_{p,b,c}\left( \frac{y}{x+a+\sqrt{x^{2}+2ax}}\right)\, dx \\
  &= (1+\lambda+p)\,2^{-\mu-p}\,a^{\mu-1-\lambda-p}\,y^{p+1}\,
     \frac{\Gamma (2 \mu)\, \Gamma (1+\lambda+p-\mu)}{\Gamma \left(\frac{3}{2}\right)\,\Gamma (2+\lambda+p+\mu)\,
     \Gamma \left(1+\frac{b}{2}+p\right) } \\
   & \times {}_4F_5 \left[\begin{array}{r}
                         \frac{3}{2} + \frac{\lambda}{2}+\frac{p}{2},\, \frac{1}{2} + \frac{\lambda}{2}+\frac{p}{2}-\frac{\mu}{2},\, 1 +  \frac{\lambda}{2}+\frac{p}{2}-\frac{\mu}{2},\,1;\,  \\
          \frac{1}{2} + \frac{\lambda}{2}+\frac{p}{2},\, 1 +  \frac{\lambda}{2}+\frac{p}{2}+\frac{\mu}{2},\,\frac{3}{2} + \frac{\lambda}{2}+\frac{p}{2}+\frac{\mu}{2},\,  \frac{1}{2}+p,\, \frac{3}{2};\,
                          \end{array}
  - \frac{y^2}{4 a^2}   \right].
  \endaligned
\end{equation}

\end{corollary}

\vskip 3mm

\begin{corollary}\label{cor4}
Let $a,\, y \in \mathbb{R}^+$.
Also, let $p,\, \lambda,\,\mu \in \mathbb{C}$ with
$\Re(\mu +p)>-1$ and $\Re(\lambda)>\Re(\mu)$.
Then
\begin{equation}\label{cor4-eq1}
\aligned
 & \int_{0}^{\infty }x^{\mu -1}\left( x+a+\sqrt{x^{2}+2ax}\right) ^{-\lambda }
\mathcal{W}_{p,b,c}\left( \frac{xy}{x+a+\sqrt{x^{2}+2ax}}\right) \,dx \\
&\hskip 5mm = 2^{-\mu-2p-1}\, a^{\mu-\lambda}\,y^{p+1}\, \Gamma (\lambda-\mu) \\
& \times {}_3 \Psi_4\left[ \aligned \left(1,\,1\right),\,\left(\lambda+p+2,\,2\right),\, \left(2\mu+2p+2,\,4\right)  \, &;\\
                      \left(\frac{3}{2},\,1\right),\,\left(p+\frac{1}{2},\,1\right),\,\left(\lambda+p+1,\,2\right),
    \,\left(\lambda+\mu +2p+3,\,4\right)    \,&; \endaligned
                  \,\, - \frac{y^2}{16} \right].
\endaligned
\end{equation}

\end{corollary}

\end{document}